\documentclass{amsart} 
\usepackage{amsfonts,amssymb,amsmath,amsthm} 
\usepackage{url} 
\usepackage{enumerate} 
\usepackage{stmaryrd}
\usepackage{color}

\newtheorem{thrm}{Theorem}[section]
\newtheorem{lem}[thrm]{Lemma}
\newtheorem{prop}[thrm]{Proposition}

\theoremstyle{definition}

\newtheorem{remark}[thrm]{Remark}
\newtheorem{example}[thrm]{Example}
\numberwithin{equation}{section}

\newcommand{\M}{\mathsf{M}}
\newcommand{\Sc}{\mathsf{S}}
\newcommand{\U}{\mathsf{U}}
\newcommand{\Gr}{\mathrm{Gr}}
\newcommand{\Vol}{\mathrm{Vol}}
\newcommand{\des}{\mathrm{des}}

\author{A.~U. Ashraf}
\address{
Mathematics Department\\
University of Western Ontario\\
London, Ontario, Canada}
\email{aashra9@gmail.com}
\thanks{This work was done while the author was a graduate student at University of Western Ontario}

\begin{document}

\title[Ahmed U.Ashraf]{An alternative approach to Volumes of Matroid Polytopes}

\begin{abstract}
    We present a new algorithm for computing the volume of an arbitrary matroid base polytope. We provide two applications of this approach: a relation between the volume of the base polytope of a matroid $\M$ and its relaxation $\M'$, and a formula for the volume of an arbitrary sparse paving matroid base polytope. 
\end{abstract}
\maketitle
\tableofcontents

\section{Introduction} \label{sec:intro}
For a point $p$ in the Grassmannian $\Gr_{k, n}(\mathbb{C})$ with a representative $k \times n$ matrix $A_p$, the closure of the torus orbit of $p$ is a toric variety denoted by $X_p$. It is known that the corresponding lattice polytope of $X_p$ is the matroid polytope, $P_\M$, of the matroid $\M$ of the matrix $A_p$. Thus, the degree of this variety is given by the volume of $P_\M$. This motivates the interest in computing the volume of a general matroid polytope. There has been plenty of work done on this problem \cite{ardila-benedetti-doker} and its special cases \cite{lam-postnikov-I}. Here we present another approach for the general case that is more amenable when the cyclic flats of the matroid are known. We assume familiarity with matroid theory and suggest the standard reference \cite{oxley-matroid} for further inquiry.

Let $\M$ be a matroid of rank $d+1$ on a set $E$ of size $n+1$. The \emph{matroid (base) polytope} $P_\M$ is defined as the convex hull of the incidence vectors of bases of $\M$; that is,
\begin{align}
    P_\M &= \mathrm{conv}\bigg(\sum_{i \in B} \mathbf{e}_i: B \in \mathcal{B}(\M)\bigg) \subseteq \mathbb{R}^E
\end{align}
where $\mathcal{B}(\M)$ is the collection of bases of $\M$. The dimension and the facets of $P_\M$ were characterised combinatorially in \cite{feichtner-sturmfels}. In particular, they showed that the dimension of $P_\M$ is $n+1-c$, where $c$ is the number of connected components of $\M$. At about the same time, Ardila and Klivans \cite{ardila-klivans} gave a combinatorial characterisation of all faces of $P_\M$. It is worth mentioning that the face lattice of $P_\M$ is still not entirely understood.

We use $\Vol(P)$ to denote the normalised volume of the polytope $P$ in its affine hull $\mathrm{aff}(P)$, fixing 1 to be the volume of the unit hypercube in $\mathrm{aff}(P)$. We use the terminology \emph{matroid base polytope} and \emph{base polytope of a matroid} interchangeably. From the definition, it is clear that the volume of a matroid polytope is a matroid isomorphism invariant. 

In the rest of the paper, we will fix our underlying set $E$ to be the set $\llbracket n \rrbracket := \{0,1,\ldots, n\}$, and assume $n \geq 1$. In what follow, we assume the rank of our matroid is $d+1$, where $d$ is a positive integer. We also note that if $\M = \M_1 \oplus \cdots \oplus \M_k$ is a decomposition of $\M$ into its connected components, then $P_\M = P_{\M_1} \times \cdots \times P_{\M_k}$. This implies that $\Vol(P_\M) = \Vol(P_{\M_1}) \cdots \Vol(P_{\M_k})$. Hence, we further reduce our discussion to the connected case.

We present a combinatorial approach to compute the volume of a (connected) matroid polytope. This approach is different from the one presented by Ardila, Benedetti and Doker in \cite{ardila-benedetti-doker}. As an application, we derive the relation between volume of a base polytope of a matroid $\M$ and its relaxation $\M'$. Using this relation inductively, we give a formula for volumes of sparse paving matroid polytopes. Note that the volume of a sparse paving matroid polytope can also be derived from a result of Joswig and Schr\"oter (Theorem 26 in \cite{joswig-schroter}) on split matroids. Our main result is primarily based on Hampe's relation (Theorem 3.12 in \cite{hampe-matroid}) between a loopless matroid $\M$ and Schubert matroids associated with $\M$. We state the main theorem below and provide its proof in Section \ref{sec:cycflat}.

\begin{thrm}\label{thrm:main}
Let $\M$ be a connected matroid on $\llbracket n \rrbracket$ of rank $d+1$. Then the normalised volume of its base polytope is given by
\begin{align}
    \mathrm{Vol}(P_\M) &= \frac{1}{n!} \sum_{ \mathcal{F} } \mu(\mathcal{F}) \delta_{\preceq}(\mathbf{b}_{\mathcal{F}}) \label{eq:main}
\end{align}
where the sum is over all anchored chains of cyclic flats.
\end{thrm}

When the lattice of cyclic flats of a given matroid is known, we can explicitly compute this volume. In the special case of connected sparse paving matroids, we can reduce the summation to a closed formula.

\begin{thrm}\label{thrm:spav}
Let $\M_\alpha$ be a connected sparse paving matroid on $\llbracket n \rrbracket$ of rank $d+1$ with $\alpha$ circuit-hyperplanes. Then
\begin{align}
    \mathrm{Vol}(P_{\M_\alpha}) &= \frac{1}{n!} \left( A_{n, d} - \alpha \binom{n-1}{d}
    \right)
\end{align}
where $A_{n, d}$ is the $(n,d)$-th Eulerian number: the number of permutations of $[n] := \{1,2,\ldots,n\}$ with $d$ descents.
\end{thrm}

We prove this in Section \ref{sec:sparelax} as a consequence of Theorem \ref{thrm:main}. Another interesting application of this result is the relation between the volume of a base polytope of a matroid $\M$ and its relaxation $\M'$.

\begin{thrm}\label{thrm:relaxation}
Let $\M$ be a matroid on $\llbracket n \rrbracket$ of rank $d+1$ with a circuit-hyperplane $H$, and let $\M'$ be the relaxation of $\M$ with respect to $H$. Then
\begin{align}
    \Vol(P_{\M'}) = \Vol(P_{\M}) + \binom{n-1}{d}  
\end{align}
\end{thrm}

This also implies that two matroid base polytopes have the same volume if they have a common relaxation (up to matroid isomorphism). In particular, the matroids $\mathsf{R}_6$ and $\mathsf{Q}_6$ both have a relaxation that is isomorphic to the matroid $\mathsf{P}_6$. This provides a structural explanation to why the base polytopes of matroids $\mathsf{R}_6$ and $\mathsf{Q}_6$ have the same volume.

\section{The Young poset on $(\llbracket n \rrbracket, d+1)$-Schubert matroids} \label{sec:young-poset}

Let us denote by $L_{d}(n)$ the set of all binary sequences of length $n+1$ with exactly $d+1$ 1s, that start with a 1 and end with a 0. We define the following partial order $\preceq$ on $L_{d}(n)$:
\begin{align}
    \mathbf{a} \preceq \mathbf{b} ~~\text{iff}~ \sum_{i=0}^j a_i \leq \sum_{i=0}^j b_i~~\forall j \in \llbracket n \rrbracket
\end{align}
Under this partial order, we have the top element $\hat{1}_{L_d(n)} = 1^{d+1} 0^{n-d}$ and the bottom element $\hat{0}_{L_d(n)} = 10^{n-d-1}1^{d}0$. Note that $L_d(n)$ is isomorphic to its dual under the involution that maps a binary sequence $\mathbf{b} = (b_0, b_1, \ldots,  b_n) \in L_d(n)$ to $\mathbf{b}^* = (b_0, b_{n-1}, b_{n-2}, \ldots, b_1, b_n)$. For a permutation $w= w_1 \ldots w_n$ of $[n]$, the \emph{binary descent sequence} $\des(w) = (d_0, d_1, \ldots, d_n)$ of $w$ is defined by
\begin{align}
    d_i &= \begin{cases}
    1~~~&\text{if}~w_{i} > w_{i+1}\\
    0~~~&\text{if}~w_{i} < w_{i+1}
    \end{cases}
\end{align}
for $i=1, \ldots, n-1$ where we set $d_0 = 1$ and $d_n=0$ for all permutations. Note that the descent sequence $\des(w)$ is an element of $L_d(n)$ for every permutation $w$. For each $\mathbf{a} \in L_d(n)$, let $\delta(\mathbf{a})$ denote the number of permutations on $[n]$ with descent sequence $\mathbf{a}$. 
\begin{lem} \label{lem:binomial}
For integers $n\geq d > 0$, we have
\begin{align}\label{eq:bin1}
\delta(1^{d+1}0^{n-d}) &= \binom{n-1}{d}
\end{align}
and
\begin{align}\label{eq:bin2}
\delta((10)^{d+1}) = E_{2d}    
\end{align}
where $E_{2d}$ is the $(2d$)th Euler number that counts the number of zigzag permutations on $[2d] := \{1, 2, \ldots, 2d\}$. 
\end{lem}
\begin{proof}
For a permutation $w$ on $[n]$, the descent sequence $\des(w) = (1^{d+1}0^{n-d})$ if
\begin{align}
    w_1 > w_2 > \cdots > w_{d} > w_{d+1} < w_{d+2} < \cdots < w_n
\end{align}
Note that this implies $w_{d+1} = 1$, and now choosing any $d$ elements from $[n]$ to write in increasing order as $w_1, w_2, \ldots, w_d$ determine the complete permutation. This implies the Equation \eqref{eq:bin1}. For Equation \eqref{eq:bin2},observe that $\des(w) = (10)^{n+1}$ implies that
\begin{align}
    w_1 < w_2 > w_3 < w_4 \cdots > w_{2d-1} < w_{2d}
\end{align}
It is well known that such permutations are counted by the Euler number $E_{2d}$ (see Proposition 1.4.3 in \cite{stanley-I}).
\end{proof}

Furthermore, we denote 
\begin{align}
    \delta_\preceq(\mathbf{b}) &:=  \sum_{\mathbf{a} \preceq \mathbf{b}} \delta(\mathbf{a}) 
\end{align}
We state the following observation as a lemma:
\begin{lem}\label{lem:eulerian}
 For $\mathbf{b} = 1^{d+1}0^{n-d}$, we have $\delta_\preceq(\mathbf{b}) = A_{n, d}$, the $(n,d)$-Eulerian number; that is, the number of permutations on $[n]$ with $d$ descents.
\end{lem}
\begin{proof}
Note that $\mathbf{b} = 1^{d+1}0^{n-d}$ is the top element of the poset $L_d(n)$, so $\delta_\preceq(\mathbf{b})$ equals the number of all permutations on $[n]$ whose descent sequence has $d+1$ ones. The number of ones in the descent sequence of a permutation $w$ is exactly one less than the number of descents of the permutation $w$. Therefore, $\delta_\preceq$ counts the number of permutations with exactly $d$ descents, which by definition equals the Eulerian number $A_{n,d}$.
\end{proof}
 With each binary sequence $\mathbf{b} = (b_0, b_1, \ldots, b_{n}) \in L_d(n)$, we can also associate a Schubert matroid $\Sc(\mathbf{b})$ of rank $d+1$ on $\llbracket n \rrbracket$, in the following manner: Starting from the empty set $\varnothing$, reading $\mathbf{b}$ from left to right, we add a coloop \emph{(coloop addition)} for each 1 and extend freely \emph{(free extension)} for each 0. We add the elements $0,1, \ldots, n$ in their natural order. This construction is also studied in \cite{billera-jia-reiner} with the roles of 0 and 1 reversed. We make the following observation:
  \begin{lem}\label{lem:Yisom}
 The set $L_{d}(n)$ under the partial order $\preceq$ is isomorphic to the Young lattice $L(\eta_{n, d})$ where $\eta_{n, d} = (\underbrace{n-d, \ldots, n-d}_{d})$ is the rectangle partition.
 \end{lem}
 \begin{proof}
 The map $\varphi: L_d(n) \longrightarrow L(\eta_{n, d})$ is defined as follows: Let $\mathbf{b} \in L_d(n)$ and consider $\mathrm{supp}(\mathbf{b}) := \{i \in \mathbb{N}: b_i \neq 0, 0 < i < n\}$. If 
 \begin{align}
     \mathrm{supp}(\mathbf{b}) = \{i_1 < i_2 < \cdots < i_d\}
 \end{align}
 Then we defined $\lambda_j := \#\{k \in \mathbf{N}: b_k = 0, i_j < k < n \}$ for $j=1, \ldots, d$. This is the number of zeros to the right of $(j+1)$st 1, minus one. We consider the decreasing sequence $(\lambda_1, \ldots, \lambda_d)$ of non-negative numbers and suppress any zeros at the end. We denote the resulting integer partition by $\lambda$ and define $\varphi(\mathbf{b}) = \lambda$. Note that the map $\varphi$, by construction, is an injective map between two equicardinal sets. To show that it is a poset isomorphism, we take two sequences $\mathbf{a}, \mathbf{b} \in L_d(n)$ such that $\mathbf{a} \preceq \mathbf{b}$. Let $(\lambda_1, \ldots, \lambda_d)$ and $(\mu_1, \ldots, \mu_d)$ be the respective decreasing sequence of non-negative integers. By construction $\lambda_j \leq \mu_j$ for all $j=1,2, \ldots, d$, therefore $\lambda$ is below $\mu$ in the Young lattice $L(\eta_{n, d})$.
 \end{proof}
The importance of the partial order $\preceq$ on $L_d(n)$  in the context of matroid polytopes
 is evident from the result of Lam and Postnikov \cite{lam-postnikov-I}: if the Schubert matroid $\Sc(\mathbf{b})$ is connected, then its base polytope is of dimension $n$ whose normalised volume is given by
\begin{align}
    \mathrm{Vol}(P_{\Sc(\mathbf{b})}) &= \frac{\delta_\preceq (\mathbf{b})}{n!} \label{eq:ScVol}
\end{align}

\begin{example}
Note that for $\mathbf{b} = 1^{d+1}0^{n-d}$, the matroid $\Sc(\mathbf{b})$ is the uniform matroid $\U_{d+1, n+1}$. In this case, we have $\delta_\preceq(\mathbf{b}) = A_{n, d}$. Equation \ref{eq:ScVol} gives the number $\frac{1}{n!} A_{n,d}$ as the normalised volume of the hypersimplex $\triangle_{d+1, n+1}$.
\end{example}

Another well-known example of Schubert matroid is the $n$th Catalan matroid $\mathsf{C}_n$ studied in \cite{ardila-catalan}. This corresponds to the case when $n+1 = 2(d+1)$, and $\mathbf{b} = (01)^{d+1}$. The normalised volume of the matroid polytope $P_{\mathsf{C}_n}$ was calculated in \cite{bidkhori-sullivant} to equal the Eulerian number $\frac{1}{n!}\frac{1}{d}A_{n,d}$.

It is worth mentioning that the duality of $L_d(n)$ is not the same as matroid duality. Recall that for a matroid $\M$ and its dual $\M^*$, we have $\Vol(P_\M) = \Vol(P_{\M^*})$. For the duality of $L_d(n)$, we have the following proposition:
  
  \begin{prop}
  The function $\delta$ is invariant with respect to the duality of $L_d(n)$; that is, for all $\mathbf{b}$, $\delta(\mathbf{b}) = \delta(\mathbf{b}^*)$. In particular, if $G(\mathbf{x}) := \{\mathbf{y} \in L_{d}(n): \mathbf{y} \preceq \mathbf{x}\}$, then 
  \begin{align}
      \Vol(P_{\Sc(\mathbf{b})}) - \Vol(P_{\Sc(\mathbf{b}^*)}) = \sum_{\mathbf{a} \in G(\mathbf{b})} \bigg( \Vol(P_{\Sc(\mathbf{a})})    -  \Vol(P_{\Sc(\mathbf{a}^*)}) \bigg) \label{eq:dual}
  \end{align}
  \end{prop}
  \begin{proof}
Recall that $\delta(\mathbf{b})$ counts the number of permutations with descent sequence $\mathbf{b}$. The \emph{reversal involution} on permutations that maps a permutation $w = w_1w_2\cdots w_n$ to $w_nw_{n-1}\cdots w_n $ maps a permutation of descent sequence $\mathbf{b}$ to one with the descent sequence $\mathbf{b}^*$. This map being an involution shows that $\delta(\mathbf{b}) = \delta(\mathbf{b}^*)$. Now notice that from Equation \eqref{eq:ScVol} it follows that $\delta(\mathbf{b}) = \delta(\mathbf{b}^*)$, can be rewritten as
  \begin{align}
      \Vol(P_{\Sc(\mathbf{b})}) - \sum_{\mathbf{a} \in G(\mathbf{b})} \Vol(P_{\Sc(\mathbf{a})}) =  \Vol(P_{\Sc(\mathbf{b}^*)}) - \sum_{\mathbf{a} \in G(\mathbf{b}^*)} \Vol(P_{\Sc(\mathbf{a})})
  \end{align}
  This equation can then be rearranged to give Equation \eqref{eq:dual}.
  \end{proof}

\section{Cyclic flats and Schubert matroids}\label{sec:cycflat}

Recall that a cyclic flat of a connected matroid $\M$ is a flat $F$ for which the restriction $\M|F$ is connected. Note that in the case of loopless matroid the empty set $\varnothing$ and the underlying set $\llbracket n \rrbracket$ are cyclic flats. We call a chain of subsets (or flats) $\mathcal{F} : F_0 \subsetneq F_1 \subsetneq \ldots F_{k} \subseteq F_{k+1}$ of $\llbracket n 
\rrbracket$ \emph{anchored} if $F_0 = \varnothing$ and $F_{k+1} = \llbracket n \rrbracket$. Let $\mathfrak{F}_n$ be the collection of all anchored chains of subsets of $\llbracket n \rrbracket$. Given a loopless connected matroid $\M$ on $\llbracket n \rrbracket$ of rank $d+1$, we define a map $\mathfrak{F}_n \rightarrow L_d(n)$ that maps $\mathcal{F} \longmapsto \mathbf{b}_{\mathcal{F}}$ given by
\begin{align}
    \mathbf{b}_\mathcal{F} &= \underbrace{1 \ldots 1}_{\rho_\M(F_1)} \underbrace{0 \ldots 0}_{|F_1| - \rho_\M(F_1)} \ldots \underbrace{1 \ldots 1}_{(d+1) -\rho_\M(F_k)} \underbrace{0 \ldots 0}_{n+1 - \rho(F_k)} 
\end{align}
For a binary sequence $\mathbf{b}$, let $b_{i_j}$ be the $j$th occurrence of a 0 followed by a 1. We define a map $L_d(n) \rightarrow \mathfrak{F}_n$ that maps $\mathbf{b} \longmapsto \mathcal{F}_{\mathbf{b}}$, given by
\begin{align}
    \mathcal{F}_\mathbf{b}: \varnothing \subsetneq F_1 \subsetneq \cdots  \subsetneq F_{t} \subsetneq \llbracket n \rrbracket
\end{align}
where $F_{j} = \{ 0, 1, \ldots, i_j\}$. Note that, for a fixed matroid $\M$ these maps are inverse of each other.

\begin{lem}\label{lem:flagschu}
The Schubert matroid $\Sc(\mathbf{b})$ has a unique maximal chain of cyclic flats given by $\mathcal{F}_\mathbf{b}$. Furthermore, if $\M$ is a  loopless connected matroid on $\llbracket n \rrbracket$ of rank $d+1$ with a unique maximal chain of cyclic flats $\mathcal{F}$, then $\M \cong \Sc(\mathbf{b}_\mathcal{F})$.
\end{lem}
\begin{proof}
Given $\mathbf{b} = (b_0, b_1, \ldots, b_n) \in L_d(n)$, consider the flag of sets $\mathcal{F}_\mathbf{b}$. By construction of $\Sc(\mathbf{b})$ described in the paragraph, all $F_j$ in $\mathcal{F}_\mathbf{b}$ are cyclic flats. Note that $\varnothing$ and $\llbracket n \rrbracket$ are also cyclic flats of $\Sc(\mathbf{b})$. Consider any subset $S \subseteq \llbracket n \rrbracket$, and let $k$ be its maximal element with respect to the usual order on $\llbracket n \rrbracket$ . The coordinate $b_k$ is either a $0$ or $1$. If $b_k = 1$, this implies $k$ is a coloop and hence $S$ is no longer cyclic. If $b_k = 0$, let $i$ be the maximal index less than $k$ such that $b_i = 1$ and let $j$ be the minimal index greater than $k$ such that $b_j = 1$. Then $k$ belongs (only) to circuits that contain every index below $j$. Any of these circuits spans the set $\{0,1, \ldots, i, \ldots, k, \ldots, j-1 \}$, which is a cyclic flat. Hence, $S$ is one of the flats in $\mathcal{F}_\mathbf{b}$. This implies that 
\begin{align}
    \varnothing \subsetneq F_1 \subsetneq \cdots \subsetneq F_t \subseteq \llbracket n \rrbracket
\end{align}
is the unique maximal chain of cyclic flats of $\Sc(\mathbf{b})$. 

The Schubert matroid $\Sc(\mathbf{b}_\mathcal{F})$ has a unique maximal chain of cyclic flats given by $\mathcal{F}$ with the rank function prescribed by the matroid $\M$. Hence, $\M$ and $\Sc(\mathbf{b})$ have the same collection of cyclic flats with the same assigned ranks. Since the collection of cyclic flats along with their ranks uniquely determine a matroid (see \cite{brylawski}). Therefore, $\M$ and $\Sc(\mathbf{b}_\mathcal{F})$ must be isomorphic.

\end{proof}

\begin{proof}[Proof of Theorem \ref{thrm:main}]
For each anchored chain $\mathcal{F}$ of cyclic flats, we consider the polytope
\begin{align}
    P_\mathcal{F} :=  \bigg( \bigcap_{F \in \mathcal{F}} h_F^\leq \bigg) \cap \triangle_{d+1, n+1} 
\end{align}
where by definition $h_F^\leq := \{(x_0, \ldots, x_n) \in \mathbb{R}^{n+1}: \sum_{i \in F} x_i \leq \rho_\M(F)\}$, and $\rho_\M$ denotes the rank function of the matroid $\M$. We recall the halfspace-intersection description of matroid polytopes:
\begin{align}
    P_\M &=  \bigg( \bigcap_{F \in L(\M)} h_F^\leq \bigg) \cap \triangle_{d+1, n+1} 
\end{align}
where $L(\M)$ denote the lattice of flats of $\M$. We can reduce the index set of intersection to the set $L(\M)_{> 1}$ of flats of $\M$ of size strictly greater than 1. This is because the $\triangle_{d+1, n+1} \subseteq \cap_{|F| = 1} h_F^{\leq}$. Feichtner and Sturmfels \cite{feichtner-sturmfels} gave an irredundant description of $P_\M$ for a connected matroid $\M$, where the intersection varies over flats $F$ such that the restriction $\M|F$ and the contraction $\M/F$ both are connected. All such flats of size 2 or bigger are cyclic flats. In particular for Schubert matroid $\Sc(\mathbf{b}_\mathcal{F})$, by Lemma \ref{lem:flagschu}, we have
\begin{align}
P_\mathcal{F} = P_{\Sc(\mathbf{b}_{\mathcal{F}})} \label{eq:schpoly}
\end{align}
In this notation, $P_{\varnothing \subsetneq \llbracket n \rrbracket} = \triangle_{d+1, n+1}$.

The set of all anchored chains of cyclic flats of $\M$ is naturally ordered by refinement (i.e. $\mathcal{F} \leq \mathcal{G}$ if $\mathcal{F}$ refines $\mathcal{G}$), and we denote this poset by $\Gamma$. Note that the poset $\Gamma$ is a lattice with the top element $\varnothing$, the empty chain, and is without a bottom element. Let $\hat{\Gamma}$ be the lattice $\Gamma$ along with a bottom element $\hat{0}_\Gamma$ adjoined. The M\"obius function of $\Gamma$ evaluated at $\mathcal{F}$ is defined by the equation $\mu_\Gamma(\mathcal{F}) := - \mu_{\hat{\Gamma}}(\hat{0}_\Gamma, \mathcal{F})$. Let $[P_\mathcal{F}]$ be the indicator function of the polytope $P_\mathcal{F}$; that is,
\begin{align}
    [P_\mathcal{F}]: \mathbb{R}^{n+1} &\longrightarrow \mathbb{R} \\
    [P_\mathcal{F}](\mathbf{x}) &= \begin{cases}
    1 ~~~&\text{if}~\mathbf{x}\in P_\mathcal{F},\\
    0~~~&\text{otherwise}
    \end{cases}
\end{align}
We consider the M\"obius sum $\chi_\M$ of these indicator functions over $\Gamma$:
\begin{align}
    \chi_\M = \sum_{\mathcal{F} \in \Gamma} \mu(\mathcal{F})[P_\mathcal{F}] \label{eq:altsum}
\end{align}
It lives in the $\mathbb{R}$-vector space generated by indicator functions of polytopes (see \cite{lawrence}). Evaluation of $\chi_\M$ on points in $P_\M$ amounts to inclusion-exclusion over the Schubert matroids whose intersection is $P_\M$. Therefore, the sum $\chi_\M$ equals $1$ at points inside $P_\M$. 

We also show that the evaluation of $\chi_\M$ on points in $\triangle_{n+1, d+1} \backslash P_\M$ gives 0. First note that if $\mathcal{F} \leq \mathcal{G}$, then $P_\mathcal{F} \subseteq P_\mathcal{G}$. This is because we get $P_\mathcal{G}$ from $P_\mathcal{F}$ be relaxing some of the half-space intersection conditions. This implies that $[P_\mathcal{F}](\mathbf{x}) = 1$ for some $\mathbf{x} \in \mathbb{R}^d$ implies $[P_\mathcal{G}](\mathbf{x}) = 1$ for all $\mathcal{G} \geq \mathcal{F}$. So the sum on the right-hand side in Equation \eqref{eq:altsum} reduces to the M\"obius sum over a proper order filter of $\Gamma$. 

Let us denote the order filter by $R$. The sum $\chi_\M(\mathbf{x})$ for $\mathbf{x} \in \triangle_{n+1, d+1} \backslash P_\M$ equals $ \sum_{\mathcal{F} \in R} \mu(\mathcal{F})$. Hampe showed that this sum equals $0$ (see Equation (1) and Equation (2) of \cite{hampe-matroid}). Now note that $\chi_\M$ is zero outside the hypersimplex $\triangle_{n+1, d+1}$. Integrating $\chi_\M$ with respect to the Jordan measure gives
\begin{align}
    \Vol(P_\M) &= \sum_{\mathcal{F} \in \Gamma} \mu(\mathcal{F}) \frac{\delta_\preceq(\mathbf{b}_\mathcal{F})}{n!}
\end{align}
\end{proof}

\begin{remark}
This should be seen in contrast to Theorem 3.3 in \cite{ardila-benedetti-doker}, where the sum is over all the ordered collections of coconnected flats satisfying the \emph{Dragon marriage condition} (due to \cite{postnikov}). 

We recall that a coconnected flat of $\M$ is a flat $F$ such that the contraction $\M/F$ is connected. In lattice-theoretic terms, a flat $F$ is cyclic if the interval $[\varnothing, F]$ is connected (as a subposet of lattice of flats), and a flat $F$ is coconnected if the interval $[F, \llbracket n \rrbracket]$ is connected. It would be interesting to know further the relation between Theorem \ref{thrm:main} and Theorem 3.3 of \cite{ardila-benedetti-doker}.
\end{remark}

\section{Matroid relaxation and sparse paving matroid polytopes} \label{sec:sparelax}
Given a representable matroid $\M$ as a collection of vectors in a vector space over $\mathbb{F}$. A hyperplane in $\M$ might exists as a consequence of this representation. Relaxing such a hyperplane might result in a non-representable matroid (over $\mathbb{F}$). Examples of matroids obtained via this construction include the non-Fano matroid $\mathsf{F}_7^-$ (from the Fano matroid), the non-Pappus $\mathsf{N}_9^-$ (from the Pappus matroid). A hyperplane $H$ of a matroid $\M$ is called a \emph{circuit-hyperplane} if it is also a circuit (of $\M$).

\begin{proof}[Proof of Theorem \ref{thrm:relaxation}]
Note that if $\M'$ is a relaxation of $\M$ with respect to $H$, then $\Gamma(\M')$ consists of exactly one more element than $\Gamma(\M)$, namely $H$. Furthermore, $H$ is incomparable to every element of $\Gamma(\M)$ except $\varnothing$. Hence, M\"obius function of $\Gamma(\M')$ have the same value as M\"obius function of $\Gamma(\M)$, except at $\varnothing$ and $\{H\}$. Therefore using Equation \eqref{eq:main}, we will have
\begin{align}\label{eq:comp}
    \Vol(P_{\M'}) - \Vol(P_\M) &= \Vol(P_{\varnothing}) - \Vol(P_{\{ H\} })\\
    &= \frac{1}{n!} \bigg[ \delta_\preceq(1^{d+1}0^{n-d}) - \delta_\preceq(1^d010^{n-d-1})  \bigg]
\end{align}
Since $\delta_\preceq(1^{d+1}0^{n-1}) = \delta_\preceq(1^d010^{n-d-1}) + \delta(1^{d+1}0^{n-d})$, the above expression simplifies as
\begin{align}
    \Vol(P_{\M'})  - \Vol(P_\M) &= \frac{\delta(1^{d+1}0^{n-d})}{n!} 
\end{align}
\end{proof}

Theorem \ref{thrm:relaxation} also implies that if two matroids have isomorphic relaxations, then the volumes of their base polytopes are equal. This is the case with the matroids $\mathsf{R}_6$ and $\mathsf{Q}_6$. It also implies that if two matroids are relaxations of the same matroid, then the volumes of base polytopes of these relaxations are equal. This is the case with the matroids $\mathsf{R}_8$ and $\mathsf{F}_8$.

 A matroid $\M$ of rank $d+1$ on the ground set $\llbracket n \rrbracket$ is called \emph{sparse paving} if each subset of $\llbracket n \rrbracket$ of size $d+1$ is either a basis or a circuit-hyperplane.

\begin{proof}[Proof of Theorem \ref{thrm:spav}]
Since the only cyclic flats of $\M_\alpha$ are the circuit-hyperplanes, say $\{H_i\}_{i=1}^\alpha$. The lattice of cyclic flats of $\M_\alpha$ is a rank $2$ lattice with atoms $H_i$'s. Now using Equation \eqref{eq:main}, we get
\begin{align}
    \Vol(P_{\M_\alpha}) &= \frac{1}{n!} \bigg[ \sum_{i=1}^\alpha  \delta_\preceq (1^d010^{n-d-1}) - (\alpha-1) \delta_{\preceq}(1^{d+1}0^{n-d}) \bigg]
\end{align}
Since $\delta_\preceq(1^{d+1}0^{n-1}) = \delta_\preceq(1^d010^{n-d-1}) + \delta(1^{d+1}0^{n-d})$, the above expression simplifies to
\begin{align}
    \Vol(P_\M) &= \frac{1}{n!} \bigg[ - \alpha \delta(1^{d+1}0^{n-d}) + \delta_\preceq(1^{d+1}0^{n-d}) \bigg]
\end{align}
By Lemma \ref{lem:binomial} and Lemma \ref{lem:eulerian}, we have
\begin{align}
    &= \frac{1}{n!} \bigg[ -\alpha \binom{n-1}{d} + A_{n,d}\bigg]
\end{align}
\end{proof}
This proof can also be seen as an inductive application of Theorem \ref{thrm:spav}. Since we can take a connected sparse paving matroid polytope with $\alpha$ circuit-hyperplanes, and relax each circuit-hyperplane one by one, to get a uniform matroid.

\begin{example}\label{eg:MK4}
Consider $\M({K_4})$, the complete graphic matroid on $\llbracket 5 \rrbracket$. It is a connected sparse paving matroid of rank $d+1 = 3$ with $\alpha= 4$ hyperplanes. In this case, the volume is given by
\begin{align}
    \Vol(P_{\M(K_4)}) &= \frac{1}{5} (A_{5,2} - 4\binom{4}{2})\\
    &= \frac{1}{5}(66 - 4\cdot 6) = \frac{42}{5!}
\end{align}
\end{example}

\begin{remark}\label{rem:van}
Connected sparse paving matroids are split matroids in the sense of \cite{joswig-schroter}. They showed that the split flacets of such a matroid $\M$ are the same as the faces defined by cyclic flats. These faces can be written as $h_F \cap \triangle_{d+1, n+1}$ where $F$ is a cyclic flat. This implies that we get the polytope $P_\M$ by peeling off the polytopes $h_F^{\geq} \cap \triangle_{d+1, n+1}$, each of which has volume $\frac{1}{n!}\binom{n-1}{d}$. Therefore, their result also implies Theorem \ref{thrm:spav},
\end{remark}

\proof[Acknowledgements] \label{ack}
I would like to thank my supervisor, Graham Denham for his encouragement and helpful comments. And also for pointing out to me papers of Ardila, Benedetti and Doker \cite{ardila-benedetti-doker}, and Ardila, Fink and Rincon \cite{ardila-fink-rincon}. I would also like to thank Federico Ardila for encouragement and guiding comments. We also thank Jorn van der Pol for pointing out the paper of Joswig and Schr\"oter \cite{joswig-schroter} along with the Remark \ref{rem:van}.

\end{document}